\documentclass[a4paper]{article}
\usepackage{latexsym}
\usepackage{amsmath,amssymb,amsfonts}
\usepackage{xcolor}
\usepackage{multirow}
\usepackage{makecell}

\newtheorem{prethm}{{\bf Theorem}}

\newenvironment{thm}{\begin{prethm}{\hspace{-0.5
               em}{\bf.}}}{\end{prethm}}

\newtheorem{prepro}[prethm]{Proposition}

\newtheorem{prelem}[prethm]{Lemma}
\newenvironment{lem}{\begin{prelem}{\hspace{-0.5
               em}{\bf.}}}{\end{prelem}}

\newtheorem{precor}[prethm]{Corollary}
\newenvironment{cor}{\begin{precor}{\hspace{-0.5
               em}{\bf.}}}{\end{precor}}

\newtheorem{preremark}{{\bf Remark}}

\newtheorem{preexample}{{\bf Example}}

\newenvironment{example}{\begin{preexample}\em{\hspace{-0.5
               em}{\bf.}}}{\end{preexample}}

\newtheorem{predefinition}{{\bf Definition}}

\newenvironment{definition}{\begin{predefinition}\em{\hspace{-0.5
               em}{\bf.}}}{\end{predefinition}}

\newtheorem{preproof}{{\bf Proof.}}

\newenvironment{proof}[1]{\begin{preproof}{\rm
               #1}\hfill{$\Box$}}{\end{preproof}}

\input{epsf}

\author{{\normalsize { N. Ghareghani${}^{ \textrm{a},\,1}$}, {P. Sharifani ${}^{ \textrm{b}}$}\,
}\vspace{2mm} \\
{\footnotesize{$^{ \textrm{a}}$\it
Department of
Industrial design, College of Fine Arts, University of Tehran,}}\\{\footnotesize{\it   Tehran,
Iran}}\\
{\footnotesize{$^{ \textrm{b}}$\it
School of Mathematics, Institute for Research in Fundamental Sciences {\rm(IPM),}}}\\{\footnotesize{\it   Tehran,
Iran}}\\
\\{\footnotesize Emails: ghareghani@ut.ac.ir, ghareghani@ipm.ir,  pouyeh.sharifani@gmail.com.}
 }
\title{\bf\ On square factors and critical factors of $k$-bonacci words on infinite alphabet}
\date{}
\begin{document}

\maketitle
\footnotetext[1]{\tt Corresponding author}




\begin{abstract}
For any integer $k>2$, the infinite $k$-bonacci word $W^{(k)}$,  on the infinite alphabet is defined as the fixed point of the morphism
$\varphi_k:\mathbb{N}\rightarrow \mathbb{N}^2 \cup \mathbb{N}$, where
\begin{equation*}
\varphi_k(ki+j) = \left\{
\begin{array}{ll}
(ki)(ki+j+1) & \text{if } j = 0,\cdots ,k-2,\\
(ki+j+1)& \text{if } j =k-1.
\end{array} \right.
\end{equation*}
The finite $k$-bonacci word $W^{(k)}_n$ is then defined as the prefix of $W^{(k)}$ whose length is the $(n+k)$-th $k$-bonacci number.
We obtain the structure of all square factors occurring in $W^{(k)}$. Moreover, we prove that the critical exponent of  $W^{(k)}$ is $3-\frac{3}{2^k-1}$. Finally, we provide all critical factors of  $W^{(k)}$.
\end{abstract}

\vspace{3mm}
\noindent{\em Keywords}: k-bonacci words, words on infinite alphabet, square, critical exponent, critical factor.

\section{Introduction}
The infinite Fibonacci word and finite Fibonacci words are well-studied in the literature and satisfy several extremal properties, see
 \cite{cassaigne2008extremal,de1981combinatorial,mignosi1998periodicity,diekert2010weinbaum,pirillo1997fibonacci}. The infinite Fibonacci word $F^{(2)}$ is the unique fixed point of the binary morphism $0\rightarrow 01$ and $1\rightarrow 0$. The $n$-th finite Fibonacci word $F_n^{(2)}$  is the prefix of of length $f_{n+2}$ of $F^{(2)}$, where $f_n$ is the $n$-th Fibonacci number.
 A natural generalization of Fibonacci words are $k$-bonacci words which are defined on the $k$-letter alphabet $\{0,1,\ldots, k-1\}$. The infinite $k$-bonacci
 word $F^{(k)}$ is the unique fixed point of the morphism  $\phi_k(0)=01, \phi_k(1)=02,\ldots , \phi_k(k-2)=0(k-1), \phi_k(k-1)=0$ (see \cite{tan2007some}). The $n$-th finite $k$-bonacci word $F_n^{(k)}$ is defined to be $\phi^n_k(0)$ or equivalently, the prefix of length $f_{n+k}^{(k)}$ of $F^{(k)}$, where $f_{n+k}^{(k)}$ denotes the $(n+k)$-th $k$-bonacci number. While the Fibonacci words are good examples of binary words, $k$-bonacci words are good examples of words over $k$-letter alphabet and they have many interesting properties (see \cite{tan2007some, adamczewski2003balances, bvrinda2014balances, glen2006sturmian}).

 In \cite{zhang2017some}, authors defined the infinite Fibonacci word on infinite alphabet ${\mathbb N}$ as the fixed point of the morphism $\varphi_2: (2i)\rightarrow (2i)(2i+1)$ and $\varphi_2: (2i+1)\rightarrow (2i+2)$. We denote the infinite Fibonacci word on infinite alphabet by $W^{(2)}$. The $n$-th finite Fibonacci word $W_n^{(2)}$ is then defined similar as $F_n^{(2)}$. It is trivial that if digits (letters) of  $W^{(2)}$ are computed mod $2$, then the resulting word is the ordinary infinite Fibonacci word  $F^{(2)}$. Zhang et al. studied some properties of word  $W^{(2)}$. They studied the growth
order and digit sum of $W^{(2)}$ and gave several decompositions of $W^{(2)}$ using singular words.
 Glen et al.
considered  more properties of $W^{(2)}$ \cite{glen2019more}. Among other results, they investigated the structure of palindrome factors and square factors of $W^{(2)}$.In \cite{ghareghani2019arxive}, authors introduced the finite (infinite) $k$-bonacci word over infinite alphabet, for $k>2$. The $n$-th finite (res. infinite) $k$-bonacci word over infinite alphabet is denoted by $W^{(k)}$ (resp. $W_n^{(k)}$). They studied some properties of these words and classified all palindrome factors of $W^{(k)}$, for $k\geq 3$.

For a finite word $W$ and a positive integer $n$, $W^n$ is simply obtained by concatenating the word $W$, $n$ times with itself and $W^{\omega}$ is defined
as the concatenation of $W$ with itself, infinitely many times; That is $W^{\omega}=W.W.W \ldots$.
For a rational number $r$ with $r.|W|\in \mathbb{N}$, the fractional power $W^r$ is defined to be
 the prefix of length $r.|W|$ of the infinite word $W^{\omega}$.
For example if $W= 0102$ then $W^{\frac{5}{2}}= 0102010201$.
 The index of a factor $U$ of word $W$ is defined as
$$\rm{INDEX}(U,W) = \max\{r \in\mathbb{Q}: U^r\prec W \}.$$
Then the {\it critical exponent} $E(W)$ of an infinite word $W$ is given by
$$E(W)=\sup \{\rm{INDEX}(U,W) : U\in F(W)\setminus \{\epsilon\} \}.$$
A word $U$ is a critical factor of $W$ if $E(W)=\rm{INDEX}(U,W).$
The study of the existence of a factor of the form $U^r$ in a long word and specially computing the critical exponent of a long word is the subject of many papers for example see \cite{mignosi1989infinite, vandeth2000sturmian, berstel1999index, carpi2000special,damanik2002index,justin2001fractional, blondin2007critical}.
Specially, in the case of infinite $k$-bonacci word $F^{(k)}$, it is proved that $E(F^{(k)})=2+\frac{1}{\alpha_k-1}$ (see \cite{glen2009episturmian}), where
$\alpha_k$, the $k$-th generalized golden ratio, is the (unique) positive
real root of the $k$-th degree polynomial $x^k-x^{k-1}-\ldots -x-1$. It is proved that $2-\frac{1}{k}< \alpha_k<1$
\cite{dresden2014simplified, hare2014three}. Hence, $3<E(F^{(k)})<3+\frac{1}{k-1}$, and $E(F^{(2)})=2+\frac{\sqrt{5}+1}{2}$.

In this work we first investigate some properties of $W_n^{(k)}$. Then, using them, we explore the structure of all square factors
of $W_n^{(k)}$. More precisely, we prove that all square factors of $W^{(k)}$ are of the form $ki\oplus C^j(W_n^{(k)})$, for some integers $i>0$ and $j\geq 0$, where $C^j(U)$ denoted the $j$-th conjugate of word $U$. Finally, using the structure of square factors of $W^{(k)}$, we prove that the critical exponent of $W^{(k)}$ is $3-\frac{3}{2^k-1}$.



\section{Preliminaries}
In this section we give more definitions and notations that are used in the paper.
We denote the alphabet, which is a finite or countable infinite set, by ${\mathcal A}$. When ${\mathcal A}$ is a countable infinite set, we simply take
${\mathcal A}={\mathbb N}$; Then each element of ${\mathcal A}$ is called a digit (instead of a letter).
 We denote by ${\mathcal A}^*$ the set of finite words over
${\mathcal A}$ and we let ${\mathcal A}^+ ={\mathcal A}^* \setminus \{\epsilon\}$, where $\epsilon$ the empty word. We denote by ${\mathcal A}^{\omega}$ the set of all infinite words over ${\mathcal A}$ and we let ${\mathcal A}^{\infty} ={\mathcal A}^*\cup {\mathcal A}^{\omega}$.
If $a\in{\mathcal A}$ and $W \in {\mathcal A}^{\infty}$, then the symbols $|W|$ and ${|W|}_a$ denote the length of $W$, and the number of
occurrences of letter $a$ in $W$, respectively.

For a finite word $W=w_1 w_2 \ldots w_{n}$, with $w_i \in {\mathcal A}$ and for $1\leq j\leq j'\leq n$, we denote $W[j,j']=w_j \ldots w_{j'}$, and for simplicity we denote $W[j,j]$ by $W[j]$.
Let $U_i\in {\mathcal A}^{*}$, for $1 \leq i \leq n$, then $\prod_{i=n}^{1} U_i$ is defined to be $U_n U_{n-1} \ldots   U_1$.
For a finite word $W$ and an integer $n$,  $n \oplus W$ denotes the word obtained by adding $n$ to each digit of $W$. For example, let $W=01020103$ and $n=5$, then $n \oplus W=56575658$. Similarly, if every digit of $W$ is grater than $n-1$, then
$W \ominus n$ denotes the word obtained by subtracting $n$ from each digit of $W$.

A word $V\in {\mathcal A}^+$
is a factor of a word $W \in {\mathcal A}^{\infty}$, if there exist
$U\in {\mathcal A}^*$ and $U' \in {\mathcal A}^{\infty}$, such that $W=UVU'$. Similarly, a word $V\in {\mathcal A}^{\infty}$ is a factor of
$W \in {\mathcal A}^{\infty}$ if there exists $U\in {\mathcal A}^*$ such that $W=UV$. When $V$ is a factor of $W$ then we denote it as $V\prec W$.
A word $V\in {\mathcal A}^+$ (resp. $V\in {\mathcal A}^{\infty}$) is said to
be a {\textit prefix} (resp. {\textit suffix}) of a word $W\in {\mathcal A}^{\infty}$, denoted as $V\lhd W$ (resp.
$V\rhd W$), if there exists $U\in {\mathcal A}^{\infty}$ (resp. $U\in {\mathcal A}^{*}$) such that $W=VU$
(resp. $W=UV$).
  If $W\in {\mathcal A}^{*}$ and $W=VU$ (resp. $W=UV$,) we  write
$V=WU^{-1}$ (resp. $V=U^{-1}W$).
The set of all factors of a word $ w $ is denoted by $ F(w) $.
If $W=w_1\ldots w_n$ be a finite word and $0\leq j\leq n-1$, then the {\it $j$-th conjugatae} of $W$ is defined as $C^j(W)=w_{j+1}\ldots w_nw_1 \ldots w_j$. For example the word $0130102$ is the $4$-th conjugate of $0102013$.
A word $ V $ is a {\it conjugate} of $ W $ if
there exists  $0\leq j\leq n-1$ such that $V= C^j(W)$.
A factor of the form $UU$  in $W$ is called a square factor or simply a square. For a square factor
$UU=W[t,t+2|u|]$ of $W$, the {\it center} of the square $UU$ in $W$ is defined to be
$c_s(U^2,W)=t+|U|+\frac{1}{2}$.

The $n$-th $k$-bonacci number defined as

\begin{equation}\label{defkbonum}
f_n^{(k)} = \left\{
\begin{array}{ll}
0\;\; &\text{if } \,\, n = 0,\cdots ,k-2, \\
1\;\; &\text{if } \,\, n = k-1, \\
\sum_{i=n-1}^{n-k}f_i^{(k)}\;\; &\text{if}  \,\, n\geq k.
\end{array} \right.
\end{equation}

The finite (resp. infinite) $k$-bonacci words $W^{(k)}_n$ (resp. $W^{(k)}$) on infinite alphabet $\mathbb{N}$ is defined in \cite{ghareghani2019arxive}, using the morphism $\varphi_k$ given below
\begin{equation*}
\varphi_k(ki+j) = \left\{
\begin{array}{ll}
(ki)(ki+j+1)\;\; & \text{if } j = 0,\cdots ,k-2 \\
(ki+j+1)& \text{otherwise } .
\end{array} \right.
\end{equation*}
More precisely, $W^{(k)}_n=\varphi_k^n(0)$ and $W^{(k)}=\varphi_k^{\omega}(0)$ (Note that $W^{(k)}_0=F^{(k)}_0=0$).
For a fixed value of $k$, the $k$-bonacci words over infinite alphabet are reduced to $k$-bonacci words over finite alphabet when the digits are calculated $\mod\; k$. It is easy to show that for $n \geq 0$,
\begin{equation}\label{size}
|F^{(k)}_n|= |W_n^{(k)}|=f_{n+k}^{(k)}.
\end{equation}



\section{Some properties of $W_n^{(k)}$}
In this section we provide some basic properties $W_n^{(k)}$, some of which are proved in \cite{ghareghani2019arxive}. All of these properties are useful for the rest of the work.
\begin{lem} {\rm[Lemma 4 of \cite{ghareghani2019arxive}]}\label{00}
    Let $n \geq 0$ and $k>2$. The finite word $W_n^{(k)}$ contains no factor $00$.
\end{lem}

Following two lemmas give recursive formulas for computing $W^{(k)}_n$.
\begin{lem} {\rm[Lemma 5 of \cite{ghareghani2019arxive}]} \label{struct1}
	For $1 \leq n \leq k-1$,
	\begin{equation}\label{struct1eq}
		W^{(k)}_n=\prod_{i=n-1}^{0}W^{(k)}_i \, \, n.
	\end{equation}
\end{lem}

\begin{lem} {\rm[Lemma 7 of \cite{ghareghani2019arxive}]}\label{struct}
	For $n \geq k$,
	\begin{equation}\label{structeq}
 W^{(k)}_n=\prod_{i=n-1}^{n-k+1}W^{(k)}_i \, (k\oplus W^{(k)}_{n-k}).
	\end{equation}
\end{lem}

The following corollary is a direct consequence of Lemmas \ref{struct1} and  \ref{struct} and can be proved using induction on $i$.
\begin{cor}\label{corki}
	Let $i$ and $n$ be two non-negative integers, then $W^{(k)}_n\oplus ki \prec W^{(k)}_{n+ki}$.
\end{cor}

Considering the recurrence relations (\ref{struct1eq}) and  (\ref{structeq}) we have the following definitions which are very useful in the next sections.

\begin{definition}\label{defborder}
	Let $j$ be a nonnegative integer, then a factor $A$ of $W^{(k)}_n$ is called a {\it bordering factor of type $j$}, for some $n-k+1\leq j\leq n-1$ if
 $j0 \prec A \prec W^{(k)}_jW^{(k)}_{j-1} \ldots W^{(k)}_{m}$, where $m=\max \{0, n-k+1\}$. Moreover, a {\it bordering square factor} of $W^{(k)}_n$ is a bordering factor of $W^{(k)}_n$ which is also a square.
\end{definition}

\begin{definition}\label{defstrad}
	Let $n\geq k$, then a factor $A$ of $W^{(k)}_n$ is called a {\it straddling factor} of $W^{(k)}_n$ if $A=A_1A_2$, for some nonempty words $A_1$ and $A_2$, with $A_1\rhd W^{(k)}_{n-1} \ldots W^{(k)}_{n-k+1}$ and $A_2\lhd k\oplus  W^{(k)}_{n-k}$. Moreover, if a straddling factor of $W^{(k)}_n$ is also a square, it is called an {\it straddling square factor}.
\end{definition}

\begin{lem} {\rm[Lemma 10 of \cite{ghareghani2019arxive}]}\label{lastdig}
	For any $n\geq 1$, the digit $n$ is the largest digit of $W^{(k)}_n$ and appears once at the end of this word.
\end{lem}

\begin{lem}\label{i0}
	For every integer $i<n$ we have $i0\prec W^{(k)}_n$.
\end{lem}
\begin{proof}{
		Since $i+1\leq n$, we have $W^{(k)}_{i+1}\prec W^{(k)}_n$. By Lemmas \ref{struct1} and \ref{struct}, $W^{(k)}_i W^{(k)}_{i-1}\lhd W^{(k)}_{i+1}$. Hence, $i0 \prec W^{(k)}_{i+1}\prec W^{(k)}_n$ and the result follows.
	}
\end{proof}

\begin{lem}\label{i01}
	Let $0<i<n$ and $i0= W^{(k)}_n[t,t+1]$, for some $t \in \mathbb{N}$. Then we have
	$$ W^{(k)}_n[t-|W^{(k)}_i|+1,t+1]=W^{(k)}_i0.$$
	In other words, if $i0$ appears in $W^{(k)}_n$, then this $i$ appeared as the last digit of a factor $W^{(k)}_i$ of $W^{(k)}_n$.
\end{lem}
\begin{proof}{
		We prove this by induction on $n$. If $n=2$, then $W^{(k)}_2=0102$ and in this case the only possibility for $t$ is $t=2$ and $W^{(k)}_2[1,3]=W^{(k)}_1 0=010$, as desired. We suppose that the claim is true for all $m\leq n$ we want to prove this for the case $n+1$.
		If $n+1\geq k$, then by Lemma \ref{struct} we have
		\begin{equation}
		W^{(k)}_{n+1}=\prod_{t=n}^{n-k+2}W^{(k)}_t \, (k\oplus W^{(k)}_{n+1-k}).
		\end{equation}
Let $i0$ occurs in $W^{(k)}_{n+1}$, then either $i0\prec W^{(k)}_t$, for some $n-k+2\leq t \leq n$, or $i0$ is a bordering factor of $W^{(k)}_{n+1}$.
 If $i0\prec W^{(k)}_t$, then by induction hypothesis this $i$ should be the last digit of some factor $W^{(k)}_i$ of $W^{(k)}_{n+1}$. If $i0$ is a bordering factor of $W^{(k)}_{n+1}$, then it is clear that $i$ is the end digit of a factor $W^{(k)}_i$ of $W^{(k)}_{n+1}$.
	
In the case $n+1<k$, using  similar argument as the previous case and Lemma \ref{struct1} we obtain the result.
}
\end{proof}

\begin{lem}\label{lemsuff2w}
	Let $2<k<n$ and $B\rhd W_n^{(k)}$ with $|B|=|W_{n-k}^{(k)}|+|W_{n-k-1}^{(k)}|$. Then
	\begin{itemize}
		\item [{\rm (i)}] If $n=k+1$, then $|B|_2=1$.
		\item [{\rm (ii)}] If $n>k+1$, then 	
		$|B|_0>0$.
	\end{itemize}

\end{lem}
\begin{proof}
	{\begin{itemize}
			\item [{\rm (i)}] If $n=k+1$, then by (\ref{struct1eq}), $B=2.k.(k+1)$, so the result follows.
			\item [{\rm (ii)}] If $n>k+1$, then
	by (\ref{structeq}),
		\begin{equation}\label{eqsuff2w}
		W_{n-k+1}^{(k)}(k\oplus W_{n-k}^{(k)})\rhd W_{n}^{(k)}.
		\end{equation}
		 Let $D\rhd W_{n-k+1}^{(k)}$ and $|D|=|W_{n-k-1}^{(k)}|$. Then by (\ref{eqsuff2w}), to prove the lemma it is suffices to show that
		$|D|_0>0$.
		If $k=3$, when $n=5,6$ it is clear that $|D|_0>0$. So, If $n\geq 7$, then
		by Equation (\ref{structeq}), we have
		\begin{align}\nonumber
	W_{n-2}^{(n-k+1)}=	W_{n-2}^{(k)}=& W_{n-3}^{(k)} W_{n-4}^{(k)} (k\oplus W_{n-5}^{(k)})\\
		=& W_{n-3}^{(k)} W_{n-5}^{(k)} \underbrace{W_{n-6}^{(k)} (k\oplus W_{n-7}^{(k)}) (k\oplus W_{n-5}^{(k)})}_{|W_{n-4}^{(k)}|=|D|}\label{eqsuffk=3}
		\end{align}
	By (\ref{eqsuffk=3}), it is clear that $|D|_0>0$.
	
	If $k>3$ and $n<2k-1$, then by Lemma \ref{struct1}, we have $0(n-k+1)\rhd W_{n-k+1}^{(k)}$.
	Since $n>k+1$, we $|W_{n-k-1}^{(k)}|\geq 2$ and hence, $|D|_0>0$.
	
	If $k>3$ and $n\geq 2k-1$, then by Lemma \ref{struct},
	$W_{n-2k+2}^{(k)}(k\oplus W_{n-2k+1}^{(k)})\rhd W_{n-k+1}^{(k)}$. Since $k>3$,
	$|W_{n-k+1}^{(k)}|>|W_{n-2k+2}^{(k)}|+|W_{n-2k+1}^{(k)}|$ and
	$|W_{n-2k+2}^{(k)}|_0>0$, we conclude that $|D|_0>0$.
\end{itemize}		}
\end{proof}

\begin{lem}\label{lemBnotfactW}
	Let $n,k$ and $j$ be nonnegative integers with $3\leq k\leq n$ and $n-k+3\leq j\leq n$. Then
	$B=W_{n-k+2}^{(k)}k$ is not a factor of $W_j^{(k)}$.
\end{lem}
\begin{proof}{
		If $j<k$, then
$|W_j^{(k)}|_k=0$ and so $B$ is not a factor of $W_j^{(k)}$.

Hence, we shall prove the result for $k\leq j\leq n$.		
We prove this part by bounded induction on $j$. Let $p=\max \{k, n-k+3\}$. Since, $j\geq k$ and $n-k+3\leq j$, the first step of induction is $j=p$. If $p=k$,
then the only occurrence of $k$ in $W_{j}^{(k)}$, is in its last digit, we conclude that if $B\prec W_{j}^{(k)}$, then $B\rhd W_{j}^{(k)}=W_{k}^{(k)}$.
Using Lemma \ref{struct}, we have $W_{j}^{(k)}=W_{j-1}^{(k)}\ldots W_{1}^{(k)}k$. Since $(n-k+2)k \rhd B \rhd W_{k}^{(k)}$, we provide
 $n-k+2=1$, so $n=k-1<k$, which is a contradiction.

If  $p=n-k+3$, then by (\ref{structeq}), we have
\begin{equation}\label{eq2box}
W_{n-k+3}^{(k)}={ W_{n-k+2}^{(k)}} W_{n-k+1}^{(k)}\ldots W_{n-2k+4}^{(k)} {(k\oplus W_{n-2k+3}^{(k)})}
\end{equation}
If $B\prec W_{j}^{(k)}=W_{n-k+3}^{(k)}$, then there exist integers $s$ and $t$ such that $B=W_{n-k+3}^{(k)}[s,t+1]$, $W_{n-k+3}^{(k)}[t]=n-k+2$ and $W_{n-k+3}^{(k)}[t+1]=k$.
Using Lemma \ref{lastdig} and Equation (\ref{eq2box}), either $t=|W_{n-k+2}^{(k)}|$ or
$t\geq |W_{n-k+3}^{(k)}|-|W_{n-2k+3}^{(k)}|+1$. If $t=|W_{n-k+2}^{(k)}|$, then
$W_{n-k+3}^{(k)}[t+1]=0$, which is a contradiction. If
$t\geq |W_{n-k+3}^{(k)}|-|W_{n-2k+3}^{(k)}|+1$, then using (\ref{eq2box}) and the fact that
	$|W_{n-k+3}^{(k)}|\leq 2|W_{n-k+2}^{(k)}|$ and $B=|W_{n-k+2}^{(k)}|+1$, we conclude that $s< |W_{n-k+2}^{(k)}|$. Which
	implies that $|W_{n-k+3}^{(k)}[s,t+1]|_{n-k+2}\geq 2$. But by definition of $B$ and using Lemma \ref{lastdig}, we have $|B|_{n-k+2}=1$, which is a contradict. Therefore, the first step of induction is true.

We are going to prove that $B$ is not a factor of $W_{j+1}^{(k)}$.
For contrary let $B\prec W_{j+1}^{(k)}$.
By (\ref{structeq}), we have
\begin{equation}\label{eqrecwj+1}
W_{j+1}^{(k)}= W_{j}^{(k)}\ldots W_{j-k+2}^{(k)} (k\oplus W_{j-k+1}^{(k)}).
\end{equation}

By induction hypothesis $B$ is not a factor of  $W_{i}^{(k)}$ for $j-k+1\leq i\leq j$.
Since $B$ contains the digit $0$ and $(k\oplus W_{j-k+1}^{(k)})$ does not contain it, there are two following possible cases for $B$:
\begin{itemize}
	\item {\bf Case 1.} $B$ is a bordering factor of $W_{j+1}^{(k)}$; Let ${\ell}$ be largest integer
	such that $B\prec W_{j}^{(k)}\ldots W_{\ell}^{(k)}$. Since for every integer $i$, $W_{i}^{(k)}$ start with $0$, we have $(n-k+2)k\prec W_{\ell}^{(k)}$ which means that ${\ell}> n-k+2$.
 Therefore, $W_{n-k+3}^{(k)}\lhd W_{\ell}^{(k)}$. Hence, $W_{n-k+2}^{(k)}0\lhd W_{\ell}^{(k)}$.
	Since $(n-k+2)k\prec W_{\ell}^{(k)}$, there exists integer $\alpha$ such that  $W_{\ell}^{(k)}[\alpha,\alpha+1]=(n-k+2)k$.	
	By Lemma \ref{lastdig}, $\alpha>|W_{n-k+2}^{(k)}|$. Therefore, $B\prec W_{\ell}^{(k)}$, which contradicts to the definition of bordering factor. Hence, $B$ is not a bordering factor of $W_{j+1}^{(k)}$.
	\item {\bf Case 2.} $B$ is a straddling factor of $W_{j+1}^{(k)}$; By definition of straddling factor, there exists nonempty word $S$ which is the suffix of $B$ and a prefix of $k\oplus W_{j-k+1}^{(k)}$.
	Since $j<n$, we have $j-k+2<n-k+2$ and hence using (\ref{eqrecwj+1}), we provide that the last two digits of $B$ occur in
	$k\oplus W_{j-k+1}^{(k)}$. Let $|S|=t+1$, for some $t>0$, this means that $(k\oplus W_{j-k+1}^{(k)})[t,t+1]=(n-k+2)k$.
By definition of $S$ we obtain
\begin{align}\label{eqbs-1}
		BS^{-1}\rhd& W_{j+1}^{(k)}[(k\oplus W_{j-k+1}^{(k)})]^{-1}
\end{align}
 On the other hand, since $(n-k+2)\prec k\oplus W_{j-k+1}^{(k)}$, we have $n-k+2\geq k$, or $n\geq 2k-2$. 	
 If $n=2k-2$, then $W_{j-k+1}^{(k)}[t,t+1]=00$, which contradicts to Lemma \ref{00}. Therefore, $n\geq 2k-3$, now,
	using Equations (\ref{struct1eq}) and (\ref{structeq}), and the fact that $n-k+3\leq j$, we have
	\begin{equation}\label{eqprefWj+1}
		W_{n-2k+2}^{(k)} W_{n-2k+1}^{(k)}\lhd W_{n-2k+3}^{(k)}\lhd W_{j-k+1}^{(k)}
	\end{equation}
By Lemma \ref{lastdig} and Equation (\ref{eqprefWj+1}), we conclude that either $t= |W_{n-2k+2}^{(k)}|$ or $t>|W_{n-2k+2}^{(k)}|+|W_{n-2k+1}^{(k)}| $.
	First suppose that $t=|W_{n-2k+2}^{(k)}|$. Then by (\ref{eqprefWj+1}), we have $S=k\oplus (W_{n-2k+2}^{(k)}0)$.
	Using (\ref{eqbs-1}), we provide
	\begin{align*}
		B[(k\oplus W_{n-2k+2}^{(k)})k]^{-1}\rhd& W_{j+1}^{(k)}[(k\oplus W_{j-k+1}^{(k)})]^{-1}\\
		W_{n-k+1}^{(k)}\ldots W_{n-2k+3}^{(k)}\rhd& W_{j}^{(k)}\ldots W_{j-k+2}^{(k)}
	\end{align*}
Hence, $n-2k+3=j-k+2$, or $j=n-k+1$, which contradicts to our assumption $n-k+3\leq j$.
Now, suppose that $t>|W_{n-2k+2}^{(k)}|+|W_{n-2k+1}^{(k)}| $.
	Let $D\rhd B$, and $|D|=|W_{n-2k+2}^{(k)}|+|W_{n-2k+1}^{(k)}|$, then $D\prec (k\oplus W_{j-k+1}^{(k)})$. On the other hand using Lemma \ref{lemsuff2w} either $|D|_0>0$ or  $|D|_2>0$, which is a contradiction.
\end{itemize} }
\end{proof}

\section{Squares in $W_n^{(k)}$}
In this section we give the structure of all square factors of $W_n^{(k)}$. We first prove that when $n<2k-1$, $W_n^{(k)}$ has no square factor. Then we characterize all square factors of $W^{(k)}$.

\begin{lem}\label{bordersqu}
		For two positive integers $n$ and $k$, there is no bordering square in $W^{(k)}_n$.
	\end{lem}
	\begin{proof}{
			For contrary suppose that there exists $n-k+2 \leq j \leq n-1$, for which $W^{(k)}_n$ contains a bordering square of type $j$; we denote this word by $A$. By Definition \ref{defborder}, $j0 \prec A \prec W^{(k)}_jW^{(k)}_{j-1} \ldots W^{(k)}_{n-k+1}$. Since $A$ is a square word, so $|A|_j\geq 2$. but by Lemma \ref{lastdig}, $$|W^{(k)}_jW^{(k)}_{j-1} \ldots W^{(k)}_{n-k+1}|_j=1.$$ This
 is a contradiction.}
\end{proof}

\begin{lem}\label{sq<k+1}
	If $n<k+1$, then $W^{(k)}_{n}$ contains no square factor.
\end{lem}
\begin{proof}{
 We prove this by bounded induction on $n$.
By definition $W^{(k)}_0=0$ does not contain any square. Suppose that for any integer $i$, $0\leq i\leq n<k$,
$W^{(k)}_n$ does not contain any square. For contrary suppose that $B$ is a square factor of $W^{(k)}_{n+1}$.
By (\ref{struct1eq}) and (\ref{structeq}) we have
\begin{equation}\label{wn+1sq}
	W^{(k)}_{n+1} = \left\{
\begin{array}{ll}
W^{(k)}_{n+1}=W^{(k)}_{n}W^{(k)}_{n-1}\ldots W^{(k)}_0 (n+1)\;\; &\text{if } \,\,  n+1<k, \\
W^{(k)}_{n+1}=W^{(k)}_{n}W^{(k)}_{n-1}\ldots W^{(k)}_1 (n+1)\;\; &\text{if } \,\, n+1=k.
\end{array} \right.
\end{equation}
Using induction hypothesis and Lemma \ref{bordersqu}, we provide that $B$ is a straddling square. By Definition \ref{defstrad} and Equation (\ref{wn+1sq}), $|B|_{n+1}\geq 2$, which contradicts with Lemma \ref{lastdig}.
}
	\end{proof}


\begin{lem}\label{lemcs>}
	let $n> k$ and $A^2$ be a straddling square of $W^{(k)}_{n}$. Then $c_s(A^2,W^{(k)}_{n})> |W^{(k)}_{n}|-|W^{(k)}_{n-k}|$.
\end{lem}
\begin{proof}
	{ By (\ref{structeq}), we have
		\begin{equation}\label{eqcssq1}
		W^{(k)}_n=W^{(k)}_{n-1} \ldots W^{(k)}_{n-k+2} \boxed{ W^{(k)}_{n-k+1}}(k\oplus W^{(k)}_{n-k})
		\end{equation}
	For contrary suppose that 	$c_s(A^2,W^{(k)}_{n})< |W^{(k)}_{n}|-|W^{(k)}_{n-k}|$. Let  $A^2=W^{(k)}_{n}[s_1,s_2]$. Since  $A^2$ is a straddling square of $W^{(k)}_{n}$, $s_2\geq |W^{(k)}_{n}|-|W^{(k)}_{n-k}|+1$ and  $s_1< |W^{(k)}_{n}|-|W^{(k)}_{n-k}|$.

	 If $c_s(A^2,W^{(k)}_{n})< |W^{(k)}_{n}|-|W^{(k)}_{n-k}|-|W^{(k)}_{n-k+1}|$, then using (\ref{eqcssq1}), we conclude that $B=W^{(k)}_{n-k+1}k\prec W^{(k)}_{n-1} \ldots W^{(k)}_{n-k+2}$. By Lemma
	 \ref{lemBnotfactW} for any $n-k+2\leq j\leq n-1$, $B$ is not a factor $W^{(k)}_{j}$.
Let $s<n$ be largest integer
such that $B\prec W_{n-1}^{(k)}\ldots W_{s}^{(k)}$. Since $0 \lhd W_{s}^{(k)}$, we have $(n-k+1)k\prec W_{s}^{(k)}$. Let  $W_{s}^{(k)}[\alpha,\alpha+1]=(n-k+1)k$.
By Lemma \ref{lastdig}, $s> n-k+1$.	
Therefore, $W_{n-k+1}^{(k)}0\lhd W_{n-k+2}^{(k)}\lhd W_{s}^{(k)}$ and by Lemma \ref{lastdig},  $\alpha>|W_{n-k+1}^{(k)}|$. Therefore, $B\prec W_{s}^{(k)}$, which contradicts to Lemma \ref{lemBnotfactW}. Hence,
\begin{equation}\label{eqW=UU}
|W^{(k)}_{n}|-|W^{(k)}_{n-k}|-|W^{(k)}_{n-k+1}|+1<c_s(A^2,W^{(k)}_{n})< |W^{(k)}_{n}|-|W^{(k)}_{n-k}|.
\end{equation}
This means that the center of $A^2$ happens in $W^{(k)}_{n-k+1}$ which is distinguished by a box in (\ref{eqcssq1}). We denote the first occurrences of $A$ in $A^2$ by  $A_1$ and
the last occurrence of $A$ in $A^2$ by $A_2$.
By (\ref{eqW=UU}), we conclude that there exist non-empty words $U_1$ and $U_2$, such that $W^{(k)}_{n-k+1}=U_1 U_2$ and $U_2 \lhd A_2$.
By Lemma \ref{lastdig}, $|U_1|_{n-k+1}=0$ and  the only occurrence of $n-k+1$ in $U_2$. which is its last digit.
 We conclude that $|A_2|_{n-k+1}>0$ and all digits of $A_2$ which appear after  $n-k+1$ are greater than $k-1$.
 Hence $|A_1|_{n-k+1}>0$ and all digits of $A_1$ which appear after  $n-k+1$ should be also greater than $k-1$. Since  $|U_1|_{n-k+1}=0$, we conclude that $U_1 \rhd A_1$ and all occurrence of $n-k+1$ are before the first digit of $U_1$. But this is a contradiction, because  $|U_1|_0>0$ and hence there is a digit $0$ which appears after all digit $n-k+1$ in $A_1$.
}
\end{proof}

\begin{cor}\label{corcentersqcor}
		let $A^2$ be a straddling square of $W^{(k)}_{n}$. Then for each $i\leq k-1$, $|A|_i=0$.
\end{cor}
\begin{proof}
{By Lemma \ref{lemcs>}, $c_s(A^2,W^{(k)}_{n})> |W^{(k)}_{n}|-|W^{(k)}_{n-k}|$. Therefore, using Equation \ref{structeq}, we conclude that $A\prec (k\oplus W^{(k)}_{n-k})$. This means that all digits of $A$ are greater than $k-1$, as desired.
}
\end{proof}
\begin{lem}\label{ww<w}
	Let $i<n$ and $n \geq k+1$, then $W^{(k)}_{n}$ contains no factor of the form $W^{(k)}_{i}W^{(k)}_{i}$.
\end{lem}
\begin{proof}{
		We prove this by induction on $n$. If $n=1$, then $W^{(k)}_1=01$ contains no factor $W^{(k)}_0 W^{(k)}_0=00$.
		By (\ref{structeq}), we have
		$W^{(k)}_n=\prod_{i=n-1}^{n-k+1}W^{(k)}_i \, (k\oplus W^{(k)}_{n-k}).$
		For the contrary suppose that $W^{(k)}_{i}W^{(k)}_{i}\prec W^{(k)}_n$ for some $i<n$.
		By induction hypothesis for any $j<n$, $W^{(k)}_{i}W^{(k)}_{i}$ is not a factor of $W^{(k)}_j$.
		Now, by Definitions \ref{defborder} and \ref{defstrad}, $W^{(k)}_{i}W^{(k)}_{i}$ should be either a bordering square or a straddling square.
		By Lemma \ref{bordersqu}, $W^{(k)}_n$ contains no bordering square. Therefore, $W^{(k)}_{i}W^{(k)}_{i}$ is a straddling square of $W^{(k)}_n$ which can not be occurred by Corollary \ref{corcentersqcor}. Hence, there is no factor of the form 	$W^{(k)}_{i}W^{(k)}_{i}$ in $W^{(k)}_{n}$.
	}
\end{proof}

\begin{lem}\label{lemsq<2k-1}
	If $n<2k-1$, then $W^{(k)}_{n}$ contains no square factor.
\end{lem}
\begin{proof}{For contrary suppose that there exists a straddling square $A^2$ in $W^{(k)}_{n}$.
		By definition of straddling factor $A$ contains the digit $n-k+1$. Hence by Corollary \ref{corcentersqcor}, $n-k+1\geq k$ and $n\geq 2k-1$, which is a contradiction. }
	\end{proof}
\begin{lem}\label{lemsq=2k-1}
$kk$ is the only square of $W^{(k)}_{2k-1}$.
\end{lem}
\begin{proof}{Let $A^2=W^{(k)}_{2k-1}[t,t+|A^2|]$ be a square factor of $W^{(k)}_{2k-1}$. By Lemmas \ref{lemsq<2k-1} and \ref{sq<k+1}, we conclude that for every $j< 2k-1$,
		$W^{(k)}_{j}$ contains no square factor.
		By Lemma \ref{bordersqu}, $W^{(k)}_{2k-1}$ has no bordering square.
		Hence, $A^2$ is a straddling square.
					\begin{equation}\label{eqrecW2k-1}
		W^{(k)}_{2k-1}=W^{(k)}_{2k-2} \ldots W^{(k)}_{k}(k\oplus W^{(k)}_{k-1})
		\end{equation}
		By  Corollary \ref{corcentersqcor} and definition of straddling factor of $W^{(k)}_{2k-1}$,
		$t=|W^{(k)}_{2k-1}|-|W^{(k)}_{k-1}|$. Hence, $kk\lhd A^2$. Hence, either
		$A^2=kk$ or the number of occurrences of $kk$ in $A^2$ is at least two. If the    number of occurrences of $kk$ in $A^2$ is at least two, then $kk\prec (k\oplus W^{(k)}_{k-1})$ which means that $00\prec W^{(k)}_{k-1}$, which contradicts with Lemma \ref{sq<k+1}.
		Hence, the only possibility for $A^2$ is $kk$.  }
\end{proof}

\begin{lem}\label{lemsq=strad}
	Let $A^2$ be a square of $W_m^{(k)}$. Then there exists integers $i$ and $n\leq m$ such that $A^2\ominus ki$ is a straddling square of $W_n^{(k)}$.
\end{lem}
\begin{proof}{We prove this using induction on $m$. If $m<2k-1$, then by Lemma \ref{lemsq<2k-1},  $W_m^{(k)}$ contains no square factor and the result follows. If $m=2k-1$, then by Lemma \ref{lemsq=2k-1}, $kk$ is the only square of $W_m^{(k)}$ which is a straddling square.
		If $m>2k-1$, then by (\ref{structeq}) and using the induction hypothesis and Lemma \ref{bordersqu} the result follows.
	}
\end{proof}
The following corollary is a direct consequence of Lemma \ref{lemsq=strad}.
\begin{cor}\label{corsq=strad}
	Let $A^2$ be a square of $W^{(k)}$. Then there exists integers $i$ and $n$, such that $A^2\ominus ki$ is a straddling square of $W_n^{(k)}$.
\end{cor}

\begin{lem}\label{centersq2}
	let $A^2$ be a straddling square of $W^{(k)}_{n}$. Then $$c_s(A^2,W^{(k)}_{n})\leq |W^{(k)}_{n}|-|W^{(k)}_{n-k}|+|W^{(k)}_{n-2k+1}|+\frac{1}{2}.$$
\end{lem}
\begin{proof}{
 By (\ref{structeq}), we have
\begin{align}
W^{(k)}_n&=\prod_{i=n-1}^{n-k+1}W^{(k)}_i \, (k\oplus W^{(k)}_{n-k})\\
&= \prod_{i=n-1}^{n-k+2}W^{(k)}_i (W^{(k)}_{n-k}. \ldots . W^{(k)}_{n-2k+2}) \, (k\oplus W^{(k)}_{n-2k+1})\, (k\oplus W^{(k)}_{n-k}).\label{structsqeq}
\end{align}
We remined that $k\oplus W^{(k)}_{n-2k+1} \prec k\oplus  W^{(k)}_{n-k}$. Hence, if 	$c_s(A^2,W^{(k)}_{n})> |W^{(k)}_{n}|-|W^{(k)}_{n-k}|+|W^{(k)}_{n-2k+1}|$, then
$k\oplus ((n-2k+1) W^{(k)}_{n-2k+1})\prec k\oplus W^{(k)}_{n-k}$ it means that  $ (n-2k+1) W^{(k)}_{n-2k+1}\prec  W^{(k)}_{n-k}$. Using Lemma \ref{i0}
we conclude that  $ W^{(k)}_{n-2k+1} W^{(k)}_{n-2k+1}\prec  W^{(k)}_{n-k}$ which contradicts to Lemma \ref{ww<w}.
}

\end{proof}

\begin{lem}\label{centersq3}
	let $A^2=W^{(k)}_{n}[t,t+j]$ be a straddling square of $W^{(k)}_{n}$. Then $t> |W^{(k)}_{n}|-|W^{(k)}_{n-k}|-|W^{(k)}_{n-2k+1}|$.
\end{lem}
\begin{proof}{
	For contrary suppose that 	
	$A^2=W^{(k)}_{n}[t,t+j]\prec W^{(k)}_{n}$ for some $t< |W^{(k)}_{n}|-|W^{(k)}_{n-k}|-|W^{(k)}_{n-2k+1}|$.
	By Lemma \ref{lemcs>},
	$c_s(A^2,W^{(k)}_{n})> |W^{(k)}_{n}|-|W^{(k)}_{n-k}|$. Hence,
	\begin{equation}\label{A2<k+w}
	W^{(k)}_{n}[t_1,t_2]\prec A \prec k\oplus W^{(k)}_{n-k}.
	\end{equation}
	 Where $t_1=|W^{(k)}_{n}|-|W^{(k)}_{n-k}|-|W^{(k)}_{n-2k+1}|-1$ and $t_2= |W^{(k)}_{n}|-|W^{(k)}_{n-k}|$.
By equation (\ref{structsqeq}) an definition of $t_1$ and $t_2$, we have $W^{(k)}_{n}[t_1,t_2]= (n-2k+2)(k\oplus W^{(k)}_{n-2k+1})$. Therefore, by Equation
(\ref{A2<k+w})  $(n-2k+2)(k\oplus W^{(k)}_{n-2k+1})\prec k\oplus W^{(k)}_{n-k}$, this means that $n-2k+2\geq k$. So, $(n-3k+2) W^{(k)}_{n-2k+1}\prec  W^{(k)}_{n-k}$.
Now, by Lemma \ref{i0}, we conclude that $ W^{(k)}_{n-3k+2}W^{(k)}_{n-2k+1}\prec  W^{(k)}_{n-k}$. Since, $ W^{(k)}_{n-3k+2} \lhd W^{(k)}_{n-2k+1}$ we conclude that $ W^{(k)}_{n-3k+2}W^{(k)}_{n-3k+2}\prec  W^{(k)}_{n-k}$, which is impossible by Lemma \ref{ww<w}. 		
	}
\end{proof}
	By Lemma \ref{centersq2},  $A^2 \prec k\oplus(W^{(k)}_{n-2k+1} W^{(k)}_{n-k})$. Now using Equation(\ref{structeq}), we have
\begin{equation*}
A^2 \prec k\oplus \big(W^{(k)}_{n-2k+1}  W^{(k)}_{n-k}\big)
= k\oplus \big(W^{(k)}_{n-2k+1} W^{(k)}_{n-2k+3} [(W^{(k)}_{n-2k+3})^{-1}W^{(k)}_{n-k}]\big)
\end{equation*}

\begin{definition}\label{defV}
	Let $n,k$ be two nonnegative integers with $k\geq 3$ and $n> 2k-1$. We define the word $V^{(k)}_{n}$ as follows:
	\begin{itemize}
		\item If $2k-1< n< 3k-2$, then $V^{(k)}_{n}=W^{(k)}_{n-2k}W^{(k)}_{n-2k-1}\ldots W^{(k)}_{0}$;
		\item If $n\geq 3k-2$, then $V^{(k)}_{n}=W^{(k)}_{n-2k}W^{(k)}_{n-2k-1}\ldots W^{(k)}_{n-3k+3}$.
	\end{itemize}
\end{definition}

\begin{lem}\label{lemVprefW}
	Let $n,k$ be two positive integers with $k\geq 3$ and $n>2k-1$.
	\begin{itemize}
		\item[{\rm (i)}] If $n<3k-2$, then
		\begin{align}
		W^{(k)}_{n-2k+2} &= W^{(k)}_{n-2k+1}V^{(k)}_{n}(n-2k+2),\label{eqwn-2k+2V}\\
		W^{(k)}_{n-2k+1} &= V^{(k)}_{n}(n-2k+1). \label{eqwn-2k+1V}
		\end{align}
		\item[{\rm (ii)}] If $n=3k-2$, then
		\begin{align}
		W^{(k)}_{n-2k+2} &= W^{(k)}_{n-2k+1}V^{(k)}_{n}k,\label{eqwn-2k+2V2}\\
		W^{(k)}_{n-2k+1} &= V^{(k)}_{n}0k. \label{eqwn-2k+1V2}
		\end{align}
		\item[{\rm (iii)}] If $n>3k-2$, then
		\begin{align}
		W^{(k)}_{n-2k+2} &= W^{(k)}_{n-2k+1}V^{(k)}_{n}(k\oplus W^{(k)}_{n-3k+2}),\label{eqwn-2k+2V3}\\
		W^{(k)}_{n-2k+1} &= V^{(k)}_{n}W^{(k)}_{n-3k+2}(k\oplus W^{(k)}_{n-3k+1}). \label{eqwn-2k+1V3}
		\end{align}
	\end{itemize}
\end{lem}

\begin{lem}\label{lemVunique}
	Let $n,k$ be two positive integers with $k\geq 3$ and $n> 2k-1$.
		Then $V^{(k)}_{n}$ occurs exactly once in
	$ W^{(k)}_{n-2k+1}$.
\end{lem}
\begin{proof}{If $2k-1< n< 3k-2$, then using Definition \ref{defV} and the fact that $|V^{(k)}_{n}|_{n-2k+1}=0$
		we conclude that $V^{(k)}_{n}$ occurs exactly once in
		$ W^{(k)}_{n-2k+1}$.
If $n\geq 3k-2$, then by (\ref{structeq}) we have
$$W^{(k)}_{n-2k+1}=W^{(k)}_{n-2k}\ldots W^{(k)}_{n-3k+2}(k\oplus W^{(k)}_{n-2k})=V^{(k)}_{n}W^{(k)}_{n-3k+2}(k\oplus W^{(k)}_{n-2k}).$$
  By definition of $V^{(k)}_{n}$, $|V^{(k)}_{n}|_{n-2k}=1$ and $(n-2k)0\prec V^{(k)}_{n}$. Using the facts that $|W^{(k)}_{n-3k+2}|_{n-2k}=0$ and
  $|k\oplus W^{(k)}_{n-2k}|_{0}=0$. Hence, $(n-2k)0$ occurs one time in $W^{(k)}_{n-2k+1}$. So, $V^{(k)}_{n}$ also occurs once in $W^{(k)}_{n-2k+1}$.
  }
\end{proof}

\begin{cor}\label{corV0}
	Let $n,k$ be two positive integers with $k\geq 3$. If $2k-1< n <3k-2$, then $V^{(k)}_{n}$ in $W^{(k)}_{n-2k+1}$ always is followed by digit $n-2k+1$. If $ n \geq 3k-2$, then $V^{(k)}_{n}$ in $W^{(k)}_{n-2k+1}$ always is followed by digit $0$.
\end{cor}
\begin{proof}{If $2k-1< n <3k-2$, then the result follows using Lemma \ref{lemVunique} and Definition \ref{defV}. If $n\geq 3k-2$, then by (\ref{structeq}), $V^{(k)}_{n}0\prec W^{(k)}_{n-2k+1}$. On the other hand, by Lemma \ref{lemVunique},  $V^{(k)}_{n}$ occurs exactly once in $W^{(k)}_{n-2k+1}$. Hence, the result follows.
	}
\end{proof}
\begin{lem}\label{lem3w}
Let $n,k$ be two nonnegative integers with $k\geq 3$ and $n>2k-1$. Then
\begin{equation}\label{eq3wlem}
|W^{(k)}_{n-k}|\geq 3|W^{(k)}_{n-2k+1}|.
\end{equation}	
\end{lem}
\begin{proof}{ We can check easily that (\ref{eq3wlem}), holds in the cases $k=3$ and $5\leq n\leq 8$.
		If  $k>3$ or $k=3$ and $n\geq 7$, then
using (\ref{structeq}), we get
\begin{equation}\label{eq3wproof}
|W^{(k)}_{n-k}|\geq |W^{(k)}_{n-k-1}|+|W^{(k)}_{n-k-2}|+|W^{(k)}_{n-k-3}|
\end{equation}
If $k>3$, then $n-k-3\geq n-2k+1$.  Hence, Equation (\ref{eq3wproof}) yields the inequality $|W^{(k)}_{n-k}|\geq 3|W^{(k)}_{n-2k+1}|$, as desired. If $k=3$, then
\begin{align*}
|W^{(k)}_{n-3}|=& |W^{(k)}_{n-4}|+|W^{(k)}_{n-5}|+|W^{(k)}_{n-6}|\\
=&  2|W^{(k)}_{n-5}|+2|W^{(k)}_{n-6}|+|W^{(k)}_{n-7}| \\
=&  3|W^{(k)}_{n-5}|+|W^{(k)}_{n-6}|-|W^{(k)}_{n-8}|\\
>&  3|W^{(k)}_{n-5}|.
\end{align*}

}
\end{proof}
To find all straddling squares of $W^{(k)}_{n}$ we need to give the following definition.

\begin{definition}\label{defU}
Let $n,k$ be two nonnegative integers with $k\geq 3$ and $n\geq 2k-1$. We define the word $U^{(k)}_{n}$ to be the prefix of  $W^{(k)}_{n-2k+1}W^{(k)}_{n-k}$ of size $4|W^{(k)}_{n-2k+1}|$.
\end{definition}

We note that Definition \ref{defU} is well-defined using Lemma \ref{lem3w}.

\begin{lem}\label{lemU=<}
Let $n,k$ be two nonnegative integers with $k\geq 3$ and $2k-1< n< 3k-2$.
Then \begin{equation*}
U^{(k)}_{n}(k)=W^{(k)}_{n-2k+1}W^{(k)}_{n-2k+1}V^{(k)}_{n}(n-2k+2)W^{(k)}_{n-2k+1}.
\end{equation*}
\end{lem}
\begin{proof}
{Since $W^{(k)}_{n-2k+3}\lhd W^{(k)}_{n-k}$ we have
		\begin{equation}\label{eqwn-2kwn-k,2}
		W^{(k)}_{n-2k+1}  W^{(k)}_{n-k}
		=W^{(k)}_{n-2k+1} W^{(k)}_{n-2k+3} [(W^{(k)}_{n-2k+3})^{-1}W^{(k)}_{n-k}]
		\end{equation}

If $k=3$, then applying Equation (\ref{structeq}) for $W^{(k)}_{n-2k+3}$ and using (\ref{eqwn-2k+2V}) and (\ref{eqwn-2k+1V})
 we get
\begin{align}\nonumber
W^{(k)}_{n-2k+3}=& W^{(k)}_{n-2k+2} W^{(k)}_{n-2k+1}(k\oplus W^{(k)}_{n-2k})\\
=& \underbrace{W^{(k)}_{n-2k+1}}_{|W^{(k)}_{n-2k+1}|}\underbrace{V^{(k)}_{n}(n-2k+2)}_{|W^{(k)}_{n-2k+1}|} \underbrace{W^{(k)}_{n-2k+1}}_{|W^{(k)}_{n-2k+1}|} (k\oplus W^{(k)}_{n-2k})\label{eq3W1,k=3}
\end{align}
Using Definition \ref{defU} and Equations (\ref{eqwn-2kwn-k,2}) and (\ref{eq3W1,k=3}) we conclude that
$$U^{(k)}_{n}=W^{(k)}_{n-2k+1}W^{(k)}_{n-2k+1}V^{(k)}_{n}(n-2k+2)W^{(k)}_{n-2k+1}.$$
If $k>3$, then applying Equation (\ref{structeq}) for $W^{(k)}_{n-2k+3}$
we get
\begin{equation*}
W^{(k)}_{n-2k+2} W^{(k)}_{n-2k+1} W^{(k)}_{n-2k}\lhd  W^{(k)}_{n-2k+3}\\
\end{equation*}
Hence, using (\ref{eqwn-2k+2V}) and (\ref{eqwn-2k+1V}) we provide
\begin{equation*}
\underbrace{W^{(k)}_{n-2k+1}}_{|W^{(k)}_{n-2k+1}|}\underbrace{V^{(k)}_{n}(n-2k+2)}_{|W^{(k)}_{n-2k+1}|} \underbrace{W^{(k)}_{n-2k+1}}_{|W^{(k)}_{n-2k+1}|}  W^{(k)}_{n-2k}\lhd  W^{(k)}_{n-2k+3}
\end{equation*}
Therefore, by Definition \ref{defU} and Equation (\ref{eqwn-2kwn-k,2}) we get
\begin{equation*}
U^{(k)}_{n}(k)=W^{(k)}_{n-2k+1}W^{(k)}_{n-2k+1}V^{(k)}_{n}(n-2k+2)W^{(k)}_{n-2k+1}.
\end{equation*}
}
\end{proof}

\begin{lem}\label{lemU=>}
Let $n,k$ be two nonnegative integers with $k\geq 3$ and $n> 3k-2$. Then
\begin{itemize}
\item[{\rm (i)}] If $k=3$, then $$U^{(k)}_{n}= W^{(k)}_{n-2k+1}W^{(k)}_{n-2k+1}V^{(k)}_{n}(k\oplus W^{(k)}_{n-3k+2})W^{(k)}_{n-2k+1}(k\oplus W^{(k)}_{n-3k+1}),$$
    \item[{\rm (ii)}] If $k>3$, then $$U^{(k)}_{n}= W^{(k)}_{n-2k+1}W^{(k)}_{n-2k+1}V^{(k)}_{n}(k\oplus W^{(k)}_{n-3k+2})W^{(k)}_{n-2k+1} W^{(k)}_{n-3k+1}.$$
\end{itemize}
\end{lem}
\begin{proof}
{By Equation (\ref{structeq}) we have
		\begin{equation}\label{eqwn-2kwn-k}
		W^{(k)}_{n-2k+1}  W^{(k)}_{n-k}
		=W^{(k)}_{n-2k+1} W^{(k)}_{n-2k+3} [(W^{(k)}_{n-2k+3})^{-1}W^{(k)}_{n-k}]
		\end{equation}
If $k=3$, then applying Equation (\ref{structeq}) for $W^{(k)}_{n-2k+3}$ and using (\ref{eqwn-2k+2V3}) and (\ref{eqwn-2k+1V3})
 we get
\begin{align}\nonumber
W^{(k)}_{n-2k+3}=& W^{(k)}_{n-2k+2} W^{(k)}_{n-2k+1}(k\oplus W^{(k)}_{n-2k})\\
=& \underbrace{W^{(k)}_{n-2k+1}}_{|W^{(k)}_{n-2k+1}|}\underbrace{V^{(k)}_{n}(k\oplus W^{(k)}_{n-3k+2}) W^{(k)}_{n-2k+1}(k\oplus W^{(k)}_{n-3k+1})}_{2|W^{(k)}_{n-2k+1}|} [(k\oplus W^{(k)}_{n-3k+1})^{-1}(k\oplus W^{(k)}_{n-2k})]\label{eq3W1,k=3,2}
\end{align}
Using Definition \ref{defU} and Equations (\ref{eqwn-2kwn-k}) and (\ref{eq3W1,k=3,2}) we conclude that
$$U^{(k)}_{n}=W^{(k)}_{n-2k+1}W^{(k)}_{n-2k+1}V^{(k)}_{n}(k\oplus W^{(k)}_{n-3k+2})W^{(k)}_{n-2k+1}(k\oplus W^{(k)}_{n-3k+1}).$$
If $k>3$, then applying Equation (\ref{structeq}) for $W^{(k)}_{n-2k+3}$
we get
\begin{equation}\nonumber
 W^{(k)}_{n-2k+2} W^{(k)}_{n-2k+1}W^{(k)}_{n-2k}\lhd W^{(k)}_{n-2k+3}
\end{equation}
Where,
\begin{equation}\label{eq3W1,k>3,2}
 W^{(k)}_{n-2k+2} W^{(k)}_{n-2k+1}W^{(k)}_{n-2k}=
 \underbrace{W^{(k)}_{n-2k+1}}_{|W^{(k)}_{n-2k+1}|}\underbrace{V^{(k)}_{n}(k\oplus W^{(k)}_{n-3k+2}) W^{(k)}_{n-2k+1} W^{(k)}_{n-3k+1}}_{2|W^{(k)}_{n-2k+1}|} [(W^{(k)}_{n-3k+1})^{-1} W^{(k)}_{n-2k}]
\end{equation}

Using Definition \ref{defU} and Equations (\ref{eqwn-2kwn-k}) and (\ref{eq3W1,k>3,2}) we conclude that
$$U^{(k)}_{n}=W^{(k)}_{n-2k+1}W^{(k)}_{n-2k+1}V^{(k)}_{n}(k\oplus W^{(k)}_{n-3k+2})W^{(k)}_{n-2k+1} W^{(k)}_{n-3k+1}.$$
}
\end{proof}

In the next lemma we give a formula for $U^{(k)}_{3k-2}$, the proof is similar to the proof of Lemma \ref{lemU=>}
 so it is omitted.

\begin{lem}\label{lemU==}
	Let $n,k$ be two nonnegative integers with $k\geq 3$ and $n= 3k-2$. Then
	\begin{itemize}
		\item[{\rm (i)}] If $k=3$, then $$U^{(k)}_{n}= W^{(k)}_{n-2k+1}W^{(k)}_{n-2k+1}V^{(k)}_{n}kW^{(k)}_{n-2k+1}k,$$
		\item[{\rm (ii)}] If $k>3$, then $$U^{(k)}_{n}= W^{(k)}_{n-2k+1}W^{(k)}_{n-2k+1}V^{(k)}_{n}kW^{(k)}_{n-2k+1}0.$$
	\end{itemize}
\end{lem}

\begin{cor}\label{corU}	
	If $A^2$ is a straddling square of $W^{(k)}_{n}$, then
	\begin{itemize}
 		\item[{\rm (i)}] $A^2\prec k\oplus U^{(k)}_{n}$,
		\item[{\rm(ii)}]  $c_s(A^2\ominus k, U^{(k)}_{n})\leq 2|W^{(k)}_{n-2k+1}|+\frac{1}{2}$.
\end{itemize}
\end{cor}
\begin{proof}{	According to Definition \ref{defU} we have	
\begin{itemize}
	\item[{\rm (i)}] This is the direct consequence of Lemma \ref{centersq3} and equation (\ref{structsqeq}).
	\item[{\rm (ii)}] This can be deducted easily from Lemma \ref{centersq2}.
\end{itemize}}
\end{proof}

\begin{lem}\label{lem*V*}
	Let $n< 3k-1$, then the word $(n-2k+1) V^{(k)}_{n} (n-2k+1)$ occurs exactly once in $U^{(k)}_{n}$.
\end{lem}
\begin{proof}
	{Using Lemma \ref{lemU=<} and Equation (\ref{eqwn-2k+1V}) we have
		\begin{equation}\label{eq*V*}
		U^{(k)}_{n}(k)=V^{(k)}_{n}(n-2k+1)V^{(k)}_{n}(n-2k+1)V^{(k)}_{n}(n-2k+2)V^{(k)}_{n}(n-2k+1).
		\end{equation}
By Deifinition\ref{defV}, it is clear that $0\lhd V^{(k)}_{n}$, $|V^{(k)}_{n}|_{n-2k+1}=0$  and $|V^{(k)}_{n}|_{n-2k+2}=0$.
Hence, using (\ref{eq*V*})
 we conclude that $ V^{(k)}_{n}$ occurs exactly four times in $U^{(k)}_{n}$. By Equation (\ref{eq*V*}),  the word $(n-2k+1) V^{(k)}_{n} (n-2k+1)$ occurs exactly once in $U^{(k)}_{n}$.
}
\end{proof}

\begin{lem}\label{lem*V0}
	Let $n\leq 3k-1$, then the word $(n-2k+1) V^{(k)}_{n} 0$ occurs exactly once in $U^{(k)}_{n}$.
\end{lem}
\begin{proof}
	{By Deifinition\ref{defV}, it is clear that $0\lhd V^{(k)}_{n}$, $|V^{(k)}_{n}|_{n-2k+1}=0$  and $|V^{(k)}_{n}|_{n-2k+2}=0$.
		Hence, using Lemmas \ref{lemU==} and \ref{lemU=>} and using Lemma \ref{lemVunique}, we conclude that  $ V^{(k)}_{n}$ occurs exactly four times in $U^{(k)}_{n}$. Therefore, by Equations (\ref{eqwn-2k+1V2}) and (\ref{eqwn-2k+1V3}) it is easy to see that $(n-2k+1) V^{(k)}_{n} 0$ occurs exactly once in $U^{(k)}_{n}$.
		
	}
\end{proof}

\begin{lem}\label{centersq}
	Let $n\geq 2k-1$ and $A^2$ be a straddling square of $W^{(k)}_{n}$ and let $A'^2=A^2\ominus k$. Then
	\begin{equation}\label{eqcenterU}
	c_s(A'^2, U^{(k)}_{n})\leq |W^{(k)}_{n-2k+1}|+|V^{(k)}_{n}|+\frac{1}{2}.
	\end{equation}
\end{lem}

\begin{proof}
{  For contrary suppose that $c_s(A'^2, U^{(k)}_{n})>|W^{(k)}_{n-2k+1}|+|V^{(k)}_{n}|+\frac{1}{2}$. We divide the proof in the following cases:
	\begin{itemize}
		\item If $n<3k-2$, then by Lemma \ref{lemU=<} and Equation (\ref{eqwn-2k+1V}) we conclude that
		\begin{equation}\label{eqprefU1}
	 W^{(k)}_{n-2k+1}V^{(k)}_{n}(n-2k+1)\lhd U^{(k)}_{n}.
	 \end{equation}
		 Using (\ref{eqprefU1}) and the fact that $A'^2\oplus k$ is a straddling square of $W^{(k)}_{n}$, we conclude that
		 $(n-2k+1)V^{(k)}_{n}(n-2k+1)$ should occurs at least twice in $A'^2\prec U^{(k)}_{n}$. This is a contradiction with Lemma \ref{lem*V*}.
		\item  If $n<3k-2$, then by Lemma \ref{lemU=<} and Equation (\ref{eqwn-2k+1V}) we conclude that
		\begin{equation}\label{eqprefU2}
		W^{(k)}_{n-2k+1}V^{(k)}_{n}0\lhd U^{(k)}_{n}.
		\end{equation}
		Using (\ref{eqprefU2}) and the fact that $A'^2\oplus k$ is a straddling square of $W^{(k)}_{n}$, we conclude that
		$(n-2k+1)V^{(k)}_{n}0$ should occurs at least twice in $A'^2\prec U^{(k)}_{n}$. This is a contradiction with Lemma \ref{lem*V0}.
	\end{itemize}	
}
\end{proof}

The following corollary is the direct consequence of Lemma \ref{centersq}.
\begin{cor}\label{centersq5}
	let $A^2$ be a straddling square of $W^{(k)}_{n}$. Then $c_s(A^2,W^{(k)}_{n})< |W^{(k)}_{n}|-|W^{(k)}_{n-k}|+|V^{(k)}_{n}|$.
\end{cor}

\begin{lem}\label{lemA2inP}
	Let $n\geq 2k-1$, $P^{(k)}_{n}=W^{(k)}_{n-2k+1} W^{(k)}_{n-2k+1}V^{(k)}_{n}$. Then $A^2$ is a straddling square of $W^{(k)}_{n}$ if and only if
	$A'^2=A^2\ominus k$ is a square of $P^{(k)}_{n}$ satisfying following properties:
	\begin{itemize}
		\item [{\rm(i)}]
	$	|W^{(k)}_{n-2k+1}|\leq c_s(A'^2, P^{(k)}_{n})\leq |W^{(k)}_{n-2k+1}|+|V^{(k)}_{n}|$,
		\item [{\rm(ii)}] Let $A'^2=P^{(k)}_{n}[t,t+|A'^2|]$. Then, $t< |W^{(k)}_{n-2k+1}|$.
	\end{itemize}
\end{lem}

\begin{proof}{Let $A^2$ is a straddling square of $W^{(k)}_{n}$. Then
by Corollary \ref{corU} we have $A^2\ominus k\prec U^{(k)}_{n}$. Using Lemmas \ref{lemcs>} and \ref{centersq} we conclude that
$	|W^{(k)}_{n-2k+1}|\leq c_s(A'^2, P^{(k)}_{n})\leq |W^{(k)}_{n-2k+1}|+|V^{(k)}_{n}|$. Moreover, since
$A^2$ is a straddling factor of 	$W^{(k)}_{n}$ $A'^2$ satisfying  (ii).

Now, let $A'^2$ is a square of $P^{(k)}_{n}$ which satisfies (i) and (ii), then we prove that $A^2=A'^2\oplus k$ is a straddling square of $W^{(k)}_{n}$. By Lemma \ref{lemsq<2k-1}, we conclude that $n\geq 2k-1$. Using Equation (\ref{structeq}) for $W^{(k)}_{n}$ and $W^{(k)}_{n-k+1}$ we have
\begin{align*}
W^{(k)}_{n} =& W^{(k)}_{n-1}\ldots W^{(k)}_{n-k+1}(k\oplus W^{(k)}_{n-k})\\
W^{(k)}_{n} =& W^{(k)}_{n-1}\ldots W^{(k)}_{n-k+2}
( W^{(k)}_{n-k}\ldots W^{(k)}_{n-2k+2}(k\oplus W^{(k)}_{n-2k+1}))(k\oplus W^{(k)}_{n-k})
\end{align*}
Since $W^{(k)}_{n-2k+1}V^{(k)}_{n}\lhd W^{(k)}_{n-k}$,
 we conclude that
 {\small\begin{equation}\label{eqA+kinP}
 W^{(k)}_{n} = W^{(k)}_{n-1}\ldots W^{(k)}_{n-k+2}
 W^{(k)}_{n-k}\ldots W^{(k)}_{n-2k+2}\underbrace{(k\oplus W^{(k)}_{n-2k+1} W^{(k)}_{n-2k+1}V^{(k)}_{n})}_{ P^{(k)}_{n}}(k\oplus  (W^{(k)}_{n-2k+1}V^{(k)}_{n})^{-1} W^{(k)}_{n-k})
 \end{equation}}
Using the fact that $A'^2$ satisfies (i) and (ii) and as shown in Equation (\ref{eqA+kinP}), we conclude that
$A^2=A'^2\oplus k$ is a straddling square of $ W^{(k)}_{n}$.
}	
\end{proof}

\begin{thm}\label{thmsqWn}
 Let $0\leq j\leq |V^{(k)}_{n}|$. Then $(C^{(j)} (k\oplus W^{(k)}_{n-2k+1}))^2$
		is a straddling square of  $W^{(k)}_n$. Moreover, every straddling square $A^2$ of $W^{(k)}_{n}$ is of the form
	$A^2=(C^{(j)} (k\oplus W^{(k)}_{n-2k+1}))^2$, for some $0\leq j\leq |V^{(k)}_{n}|$.
\end{thm}

\begin{proof}{By Lemma \ref{lemVprefW}, $V^{(k)}_{n}\lhd W^{(k)}_{n-2k+1}$. Hence, there exists suffix $V'$ of $W^{(k)}_{n-2k+1}$ such that $W^{(k)}_{n-2k+1}=V^{(k)}_{n}V'$. Let $0\leq j\leq |V^{(k)}_{n}|$ and $V_1= V^{(k)}_{n}[1,j]$ and $V_2= V^{(k)}_{n}[j+1,|V^{(k)}_{n}|]$. Then
		\begin{equation}\label{eqv1v2v'}
		P^{(k)}_{n}= V_1 \overbrace{ V_2 V'V_1 }^{C^{(j)} (W^{(k)}_{n-2k+1})}\underbrace{V_2 V'V_1}_{C^{(j)} (W^{(k)}_{n-2k+1})} V_2
		\end{equation}
Now, using Lemma \ref{lemA2inP} and Equation (\ref{eqv1v2v'}), 	$(C^{(j)} (k\oplus W^{(k)}_{n-2k+1}))^2$ is a straddling square of 	$W^{(k)}_{n}$.
		
Moreover, let $A^2$ be a straddling square of $W^{(k)}_{n}$. Hence the first $A$ in $A^2$ should contains $n-k+1$. By Lemma \ref{lemA2inP}  $A'^2=A^2\ominus k$ is a square factor of $P^{(k)}_{n}$	satisfying the conditions of the lemma. 	
		Therefore, $|A'|_{n-2k+1}\geq 1$. By Lemma \ref{lemA2inP}, we can assume that $c_s(A'^2,  P^{(k)}_{n})=|W^{(k)}_{n-2k+1}|+j+\frac{1}{2}$ for some $0\leq j\leq |V^{(k)}_{n}|$.

			Again using Lemma \ref{lemVprefW}, $V^{(k)}_{n}\lhd W^{(k)}_{n-2k+1}$. Let $V'\rhd W^{(k)}_{n-2k+1}$ such that $W^{(k)}_{n-2k+1}=V^{(k)}_{n}V'$, $V_1= V^{(k)}_{n}[1,j]$ and $V_2= V^{(k)}_{n}[j+1,|V^{(k)}_{n}|]$. Therefore, for the first $A'$ in $A'^2$ we have
		$A'\rhd V_1 V_2 V' V_1$ and for the last $A'$ in $A'^2$, $A'\lhd  V_2 V' V_1 V_2$. On the other hand $V_1 V_2 V' V_1 \prec W^{(k)}_{n-2k+1}V^{(k)}_{n}$ and by Definition \ref{defV} and Lemma \ref{lastdig}, $|W^{(k)}_{n-2k+1}V^{(k)}_{n}|_{n-2k+1}=1$, hence $|A'|_{n-2k+1}=1$.
		Since the first place that $n-2k+1$ occurs in $ V_2 V' V_1 V_2$ is $|V_2 V'|$, hence
		$V_2 V' \lhd A'$. Since the number of occurrences of $V_2 V'$ in $ V_1 V_2 V' V_1$ is once. We conclude that $A'= V_2 V' V_1$.
		}
\end{proof}

\begin{thm}\label{thmallsqW}
Let $k\geq 3$. Then $A^2$ is a square of $W^{(k)}$ if and only if $A\in \{ki\oplus C^j(W^{(k)}_{n-2k+1}): 0\leq j \leq |V^{(k)}_{n}|, i>0, n\geq 0\}$. 	
\end{thm}
\begin{proof}{
If $A=  ki\oplus C^j(W^{(k)}_{n-2k+1})$, for some $0\leq j \leq |V^{(k)}_{n}|, i>0, n\geq 0$, then by Theorem \ref{thmsqWn}
$k\oplus C^j(W^{(k)}_{n-2k+1})=(A^2 \ominus k(i-1))\prec W^{(k)}_{n}$ or equivalently  $A^2 \prec W^{(k)}_{n}\oplus k(i-1)$.  By Corollary \ref{corki}, we conclude that $A^2 \prec W^{(k)}_{n+k(i-1)}\prec W^{(k)}$.

On the other hand, if $A^2$ is a square of $W^{(k)}$, then by Corollary \ref{corsq=strad}, there exist $n>2k-1, i>0$, such that $A^2\ominus k(i-1)$ is a straddling square of $W_n^{(k)}$.  By Corollary \ref{corki}, we conclude that $A^2 \prec W^{(k)}_{n+k(i-1)}$.
}
\end{proof}
We finish this section with the following example.

\begin{example} In this example we provide all square factors of $W^{(3)}_{11}$, which is given bellow. All of these squares are listed in Table \ref{tableallsq} according to Theorem \ref{thmallsqW}. We note that letters $a$ and $b$ stand for the digits $10$ and $11$.
	\begin{align*}
	W^{(3)}_{11}=
	&010201301023401020133435010201301023434353460102013010234010201334353435346343567\\
	&010201301023401020133435010201301023434353463435346343567343534667680102013010234\\
	&010201334350102013010234343534601020130102340102013343534353463435673435346343567\\
	&3435346676834353463435676768679010201301023401020133435010201301023434353460102013\\
	&0102340102013343534353463435670102013010234010201334350102013010234343534634353463\\
	&4356734353466768343534634356734353466768343534634356767686793435346343567343534667\\
	&68676867967689a0102013010234010201334350102013010234343534601020130102340102013343\\
	&5343534634356701020130102340102013343501020130102343435346343534634356734353466768\\
	&0102013010234010201334350102013010234343534601020130102340102013343534353463435673\\
	&4353463435673435346676834353463435676768679343534634356734353466768343534634356767\\
	&68679343534634356734353466768676867967689a3435346343567343534667683435346343567676\\
	&8679676867967689a67686799a9b
	\end{align*}

\begin{table}\caption{Square factors of $W^{(3)}_{11}$}\label{tableallsq}
\begin{tabular}{|p{0.9in}|p{0.2in}|p{1.7in}|p{1.0in}|p{0.4in}|} \hline
{ }& \diaghead{\theadfont IIIIIIIIIIII}
	{{\;\large \textbf{j}}}{{\large \textbf{i}\;}}& 1 & 2 & 3 \\ \hline
	\vspace*{0.01 cm}$C^{(j)} (ki\oplus W^{(3)}_0)$ & \vspace*{0.01 cm}
		0& \vspace*{0.01 cm} 3 & \vspace*{0.01 cm}6 & \vspace*{0.01 cm}9 \\ \hline
	\multirow{2}{*}{$C^{(j)} (ki\oplus W^{(3)}_1)$}
	& 0 & 34 & 67 & - \\
	& 1 & 43 & 76 & - \\ \hline
	\multirow{3}{*}{$C^{(j)} (ki\oplus W^{(3)}_2)$}
	& 0 & 3435 & 6768 & - \\
	& 1 & 4353 & 7686 & - \\
	& 2 & 3534 & 6867 & - \\  \hline
	\multirow{5}{*}{$C^{(j)} (ki\oplus W^{(3)}_3)$}
	& 0 & 3435346 & 6768679 & - \\
	& 1 & 4353463 & 7686796 & - \\
	& 2 & 3534634 & 6867967 & -\\
	& 3 & 5346343 & 8679676 & - \\
	& 4 & 3463435 & 6796768 & - \\  \hline
		\multirow{8}{*}{$C^{(j)} (ki\oplus W^{(3)}_4)$}
	& 0 & 3435346343567 & - & - \\
	& 1 & 4353463435673 & - & - \\
	& 2 & 3534634356734 & - & -\\
	& 3 & 5346343567343 & - & - \\
	& 4 & 3463435673435 & - & - \\
	& 5 & 4634356734353 & - & -\\
	& 6 & 6343567343534 & - & - \\
	& 7 & 3435673435346 & - & - \\  \hline
		\multirow{8}{*}{$C^{(j)} (ki\oplus W^{(3)}_4)$}
	& 0 & 343534634356734353466768 & - & - \\
	& 1 & 435346343567343534667683 & - & - \\
	& 2 & 353463435673435346676834 & - & -\\
	& 3 & 534634356734353466768343 & - & - \\
	& 4 & 346343567343534667683435 & - & - \\
	& 5 & 463435673435346676834353 & - & -\\
	& 6 & 634356734353466768343534 & - & - \\
	& 7 & 343567343534667683435346 & - & - \\
	& 8 & 435673435346676834353463 & - & - \\
	& 9 & 356734353466768343534634 & - & - \\
	& 10 & 567343534667683435346343 & - & -\\
	& 11 & 673435346676834353463435 & - & - \\
	& 12 & 734353466768343534634356 & - & - \\
	& 13 & 343534667683435346343567 & - & -\\  \hline
\end{tabular}
\end{table}
\begin{table}
\begin{tabular}{|p{0.9in}|p{0.4in}|p{3in}|p{0.2in}|p{0.2in}|} \hline
{ }& \diaghead{\theadfont IIIIIIIIIIII}
{{\;\large \textbf{j}}}{{\large \textbf{i}\;}}& 1 & 2 & 3 \\ \hline
	
	\multirow{8}{*}{$C^{(j)} (ki\oplus W^{(3)}_4)$}
	&0 & 34353463435673435346676834353463435676768679 & - & - \\
	& 1 & 43534634356734353466768343534634356767686793 & - & - \\
	& 2 & 35346343567343534667683435346343567676867934 & - & -\\
	& 3 & 53463435673435346676834353463435676768679343 & - & - \\
	& 4 & 34634356734353466768343534634356767686793435 & - & - \\
	&5 & 46343567343534667683435346343567676867934353 & - & -\\
	& 6 & 63435673435346676834353463435676768679343534 & - & - \\
	& 7 & 34356734353466768343534634356767686793435346 & - & - \\
	& 8 & 43567343534667683435346343567676867934353463 & - & - \\
	& 9 & 35673435346676834353463435676768679343534634 & - & - \\
	& 10 & 56734353466768343534634356767686793435346343 & - & -\\
	& 11 & 67343534667683435346343567676867934353463435 & - & - \\
	& 12 & 73435346676834353463435676768679343534634356 & - & - \\
	& 13 & 34353466768343534634356767686793435346343567 & - & -\\
	& 14 & 43534667683435346343567676867934353463435673 & - & - \\
	& 15 & 35346676834353463435676768679343534634356734 & - & -\\
	& 16 & 53466768343534634356767686793435346343567343 & - & - \\
	& 17 & 34667683435346343567676867934353463435673435 & - & - \\
	& 18 & 46676834353463435676768679343534634356734353 & - & - \\
	& 19 & 66768343534634356767686793435346343567343534 & - & - \\
	& 20 & 67683435346343567676867934353463435673435346 & - & -\\
	& 21 & 76834353463435676768679343534634356734353466 & - & - \\
	& 22 & 68343534634356767686793435346343567343534667 & - & - \\
	& 23 & 83435346343567676867934353463435673435346676 & - & -\\  \hline
\end{tabular}
\end{table}
\end{example}


\newpage
\section{Critical Exponent and Critical Factors of $W^{(k)}$}
\begin{lem}\label{lempower<}
	let $A^2$ be a straddling square of $W^{(k)}_{n}$. Then
	\begin{equation*}
	\rm{INDEX}(A, W^{(k)}_{n}) = \left\{
	\begin{array}{ll}
3-\frac{1}{2^{n-2k+1}}\;\; & \text{if } 2k-1\leq n \leq 3k-3, \\
3-\frac{1}{2^{k-2}}\;\; & \text{if } n =3k-2, \\
3-\frac{|W^{(k)}_{n-3k+2}|+|W^{(k)}_{n-3k+1}|}{|W^{(k)}_{n-2k+1}|}\;\; & \text{if } n> 3k-2.
	\end{array} \right.
	\end{equation*}
\end{lem}
\begin{proof}{
		By Lemma \ref{lemA2inP}, $A^2\ominus k \prec P^{(k)}_{n}=W^{(k)}_{n-2k+1}W^{(k)}_{n-2k+1}V^{(k)}_{n}$. If $2k-1\leq n\leq 3k-3$, then by Definition \ref{defV}.
		\begin{align*}
		P^{(k)}_{n}&= W^{(k)}_{n-2k+1}W^{(k)}_{n-2k+1}V^{(k)}_{n}\\
		&= W^{(k)}_{n-2k+1}W^{(k)}_{n-2k+1}W^{(k)}_{n-2k} \ldots W^{(k)}_{0}\\
		&= W^{(k)}_{n-2k+1}W^{(k)}_{n-2k+1}W^{(k)}_{n-2k+1} (n-2k+1)^{-1}\\
		&= (W^{(k)}_{n-2k+1})^{3-\frac{1}{2^{n-2k+1}}}.
		\end{align*}
		Where the last equality holds since $|W^{(k)}_{n-2k+1}|=2^{n-2k+1}$.
		If $ n= 3k-2$, then by Definition \ref{defV}.
		\begin{align*}
		P^{(k)}_{3k-2}&= W^{(k)}_{k-1}W^{(k)}_{k-1}V^{(k)}_{3k-2}\\
		&= W^{(k)}_{k-1}W^{(k)}_{k-1}W^{(k)}_{k-2} \ldots W^{(k)}_{1}\\
		&= W^{(k)}_{k-1}W^{(k)}_{k-1}W^{(k)}_{k-1} (0(k-1))^{-1}\\
		&= (W^{(k)}_{k-1})^{3-\frac{1}{2^{k-2}}}.
		\end{align*}
	In the case $n>3k-2$, again by using Definition \ref{defV}, we conclude that $$P^{(k)}_{n}=(W^{(k)}_{n-2k+1})^{3-\frac{|W^{(k)}_{n-3k+2}|+|W^{(k)}_{n-3k+1}|}{|W^{(k)}_{n-2k+1}|}}.$$
}
\end{proof}

In the following example for $k=5$ and $9\leq n\leq 17$, we show that how Lemma \ref{lempower<} works.

\begin{example}\label{examallp}
	In this example we listed all  $P^{(5)}_{n}$, when $9\leq n\leq 17$ and for all values of $n$ we present the corresponding power $r$. We note that letters $a$ and $b$ stand for the digits $10$ and $11$.
	\begin{table}\caption{Powers of $W^{(5)}_{n-9}$ in $W^{(5)}_{n}$}\label{tableallp}
	\begin{tabular}{|p{0.2in}|p{4.2in}|p{0.4in}|} \hline
		n & $P^{(5)}_{n} \oplus k=(W^{(5)}_{n-9}\oplus k)^r $ & r \\ \hline
		9 & 55 & 2 \\ \hline
		10 & 56565 & $3-\frac{1}{2}$ \\ \hline
		11 & 56575657565 & $3-\frac{1}{4}$ \\ \hline
		12 & 56575658565756585657565 & $3-\frac{1}{8}$ \\ \hline
		13 & 5657565856575659565756585657565956575658565756 & $3-\frac{1}{8}$ \\ \hline
		14 & 565756585657565956575658565756a565756585657565956575658565756a
		5657565856575659565756585657 & $3-\frac{3}{31}$ \\ \hline
		15 & 565756585657565956575658565756a5657565856575659565756585657ab\newline 565756585657565956575658565756a5657565856575659565756585657ab\newline 565756585657565956575658565756a565756585657565956575658 & $3-\frac{6}{61}$ \\ \hline
		16 & 565756585657565956575658565756a5657565856575659565756585657ab\newline565756585657565956575658565756a565756585657565956575658abac\newline 565756585657565956575658565756a5657565856575659565756585657ab\newline565756585657565956575658565756a565756585657565956575658abac\newline 565756585657565956575658565756a5657565856575659565756585657ab\newline565756585657565956575658565756a5657565856575659 &$3-\frac{12}{120}$ \\ \hline
		17 & 565756585657565956575658565756a5657565856575659565756585657ab\newline565756585657565956575658565756a565756585657565956575658abac\newline565756585657565956575658565756a5657565856575659565756585657ab\newline565756585657565956575658565756a5657565856575659abacabad\newline 565756585657565956575658565756a5657565856575659565756585657ab\newline565756585657565956575658565756a565756585657565956575658abac\newline565756585657565956575658565756a5657565856575659565756585657ab\newline565756585657565956575658565756a5657565856575659abacabad\newline 565756585657565956575658565756a5657565856575659565756585657ab\newline565756585657565956575658565756a565756585657565956575658abac\newline565756585657565956575658565756a5657565856575659565756585657ab\newline565756585657565956575658565756a  & $3-\frac{24}{236}$ \\ \hline
	\end{tabular}
\end{table}
\end{example}

As we can see in Table \ref{tableallp} in Example \ref{examallp}, the largest
power of  $W^{(5)}_{n-9}\oplus k$ in $P^{(5)}_{n}\oplus k$ is $3-\frac{3}{31}$. This power happens when $n=14$, which is the critical exponent of  $P^{(5)}_{n}\oplus k$. Moreover,
in the following Theorem we show that
this $r$ is also the critical exponent of $W^{(5)}$.

\begin{thm}\label{thmcritical}
	Let $k\geq 3$, then the critical exponent of $W^{(k)}$ equals to $3-\frac{3}{2^k-1}$. Moreover, the set of all critical factors of
	$W^{(k)}$ is $\{P^{(k)}_{3k-1} \oplus ki\}$.
\end{thm}
\begin{proof}{
		By Theorem \ref{thmsqWn}, for all $n\geq 2k-1$, $W_n^{(k)}$ always contains a square factor. Hence $E(W^{(k)})\geq 2$.
		Let $A\in F(W^{(k)})$ and $r={\rm INDEX}(A)\geq 2$. We will prove that  $r\leq 3-\frac{3}{2^{k-1}}$. Since $r\geq 2$,
		$A^2$ is a square factor of $W^{(k)}$. By Corollary \ref{corsq=strad} there exisit integers $i$ and $n$ such that $A^2\ominus ki$ is a straddling square of $W_n^{(k)}$. Let $m_1=\max\{3-\frac{1}{2^{n-2k+1}}: 2k-1\leq n\leq 3k-3\}$, $m_2=3-\frac{1}{2^{k-2}}$ and $m_3=\max\{3-\frac{|W^{(k)}_{n-3k+2}|+|W^{(k)}_{n-3k+1}|}{|W^{(k)}_{n-2k+1}|}:  n\geq 3k-1\}$. Now, using Lemma \ref{lempower<} we conclude that $r\leq \max\{m_1, m_2, m_3\}$. It is easy to check that
		$m_1=m_2=3-\frac{1}{2^{k-2}}$.
		Since $g(n)=3-\frac{|W^{(k)}_{n-3k+2}|+|W^{(k)}_{n-3k+1}|}{|W^{(k)}_{n-2k+1}|}$ is a decreasing function of $n$, we conclude that $m_3=g(3k-1)=3-\frac{3}{2^k-1}$. Hence $r\leq 3-\frac{3}{2^k-1}$.
		On the other hand, by Lemma  \ref{lemA2inP}, $P^{(k)}_{3k-1} \oplus k=(W^{(k)}_{k})^{3-\frac{3}{2^k-1}}$.
		This implies that the set of all critical factors of  $W^{(k)}$ equals to $\{P^{(k)}_{3k-1}\oplus ki: i\geq 1\}$.
	}
\end{proof}

\begin{example} In Table \ref{tablecritical}, we compute the critical exponent and one of the critical factors of $W^{(k)}$, for $3\leq k\leq 8$, according to Theorem \ref{thmcritical}. We note that the digits $10,11, \ldots, 16$ are denoted by the letters $a, b, \ldots g$, respectively.

	\begin{table}\caption{the critical exponent and one of the critical factors of $W^{(k)}$}\label{tablecritical}
	\begin{tabular}{|p{0.2in}|p{4.4in}|p{0.8in}|} \hline
		k &  $P^{(k)}_{3k-1} \oplus k=(W^{(k)}_{k}\oplus k)^r $ & $r=3-\frac{3}{2^k-1} $\\ \hline
		3 &  343534634353463435 & $3-\frac{3}{7}$ \\ \hline
		4 & 454645474546458454645474546458454645474546 & $3-\frac{3}{15}$ \\ \hline
		5 & 565756585657565956575658565756a565756585657565956575658565756a
		5657565856575659565756585657 & $3-\frac{3}{31}$ \\ \hline
		6 & 676867696768676a676867696768676b676867696768676a67686769676867c
		676867696768676a676867696768676b676867696768676a67686769676867c
		676867696768676a676867696768676b676867696768676a676867696768 & $3-\frac{3}{63}$ \\ \hline
		7 & 7879787a7879787b7879787a7879787c7879787a7879787b7879787a7879787d
		7879787a7879787b7879787a7879787c7879787a7879787b7879787a787978e
		7879787a7879787b7879787a7879787c7879787a7879787b7879787a7879787d
		7879787a7879787b7879787a7879787c7879787a7879787b7879787a787978e
		7879787a7879787b7879787a7879787c7879787a7879787b7879787a7879787d
		7879787a7879787b7879787a7879787c7879787a7879787b7879787a7879
		& $3-\frac{3}{127}$ \\ \hline
		8 & 898a898b898a898c898a898b898a898d898a898b898a898c898a898b898a898e
		898a898b898a898c898a898b898a898d898a898b898a898c898a898b898a898f
		898a898b898a898c898a898b898a898d898a898b898a898c898a898b898a898e
		898a898b898a898c898a898b898a898d898a898b898a898c898a898b898a89g
		898a898b898a898c898a898b898a898d898a898b898a898c898a898b898a898e
		898a898b898a898c898a898b898a898d898a898b898a898c898a898b898a898f
		898a898b898a898c898a898b898a898d898a898b898a898c898a898b898a898e
		898a898b898a898c898a898b898a898d898a898b898a898c898a898b898a89g
		898a898b898a898c898a898b898a898d898a898b898a898c898a898b898a898e
		898a898b898a898c898a898b898a898d898a898b898a898c898a898b898a898f
		898a898b898a898c898a898b898a898d898a898b898a898c898a898b898a898e
		898a898b898a898c898a898b898a898d898a898b898a898c898a898b898a
		& $3-\frac{3}{255}$ \\ \hline
		\end{tabular}
\end{table}
\end{example}
\newpage
\bibliographystyle{acm}
\bibliography{mybib}

\end{document}